\documentclass{amsart}
\usepackage{amssymb, amsmath}
\usepackage{stmaryrd}
\usepackage{xcolor}

\title{Minimally critical regular endomorphisms of $\AA^N$}
\author{Patrick Ingram}
\address{York University, Toronto, Canada}

\renewcommand{\epsilon}{\varepsilon}

\newcommand{\PP}{\mathbb{P}}
\newcommand{\ZZ}{\mathbb{Z}}
\newcommand{\CC}{\mathbb{C}}
\newcommand{\RR}{\mathbb{R}}
\newcommand{\QQ}{\mathbb{Q}}
\renewcommand{\AA}{\mathbb{A}}

\newcommand{\PGL}{\operatorname{PGL}}
\newcommand{\SL}{\operatorname{SL}}

\newtheorem{theorem}{Theorem}
\newtheorem{lemma}[theorem]{Lemma}
\newtheorem{corollary}[theorem]{Corollary}
\newtheorem{prop}[theorem]{Proposition}
\newtheorem{conjecture}[theorem]{Conjecture}

\theoremstyle{definition}
\newtheorem{remark}[theorem]{Remark}
\newtheorem{question}[theorem]{Question}

\begin{document}
\maketitle

\begin{abstract} 
We study the dynamics of the map $f:\AA^N\to\AA^N$ defined by
\[f(\mathbf{X})=A\mathbf{X}^d+\mathbf{b},\]
for $A\in\SL_N$, $\mathbf{b}\in\AA^N$, and $d\geq 2$, a class which specializes to the unicritical polynomials when $N=1$. In the case $k=\CC$ we obtain lower bounds on the sum of Lyapunov exponents of $f$, and a statement which generalizes the compactness of the Mandelbrot set. Over $\overline{\QQ}$ we obtain estimates on the critical height of $f$, and over algebraically closed fields we obtain some rigidity results for post-critically finite morphisms of this form.
\end{abstract}

\section{Introduction}

The \emph{unicritical polynomials} $f(z)=z^d+c$ have been a test-bed in complex holomorphic dynamics, in part because much in dynamics is determined by the orbits of the critical points, and these polynomials have the fewest possible critical points.
Along similar lines, in studying \emph{regular polynomial endomorphisms} of $\CC^N$ (that is, polynomial maps which extend regularly to $\PP^N_\CC$) it makes sense again to consider those with the simplest possible critical locus, which in this case would consist of $N$ hyperplanes intersecting properly (ignoring multiplicity for now). After a suitable change of variables, such a map has the form \begin{equation}\label{eq:form}f(\mathbf{X})=A \mathbf{X}^d+\mathbf{b},\end{equation} with $A\in \SL_N(\CC)$ and $\mathbf{b}\in\CC^N$. We identify $f$ with its extension to $\PP^N_\CC$, and write $f|_H$ for the restriction of $f$ to the plane at infinity (which is the $d$th-power map followed by multiplication-by-$A$, a \emph{minimally critical} endomorphism of $\PP^{N-1}$ in the sense of~\cite{mincrit}).

Write $L(f)$ for the sum of Lyapunov exponents of $f$ with respect to its invariant measure. It follows from  general result of Bedford and Jonsson~\cite{bj} on regular polynomial endomorphisms that \[L(f)- L(f|_H)\geq\log d,\] and our first result is a similar lower bound which becomes arbitrarily large for certain parameters.

\begin{theorem}\label{th:proper}
For $f$ as in~\eqref{eq:form}, we have
\[L(f)- L(f|_H)\geq \frac{d-1}{d}\log^+\|\mathbf{b}\|+O_A(1)\]
and
\[L(f)- L(f|_H)\leq N(N+2)\log^+\|\mathbf{b}\|+O_A(1).\]
\end{theorem}
Explicit error terms, which are continuous and plurisubharmonic on $\SL_N(\CC)$, are given in the proof. 

Note that Favre~\cite[Theorem~C]{favre} has characterized the variation of the Lyapunov exponent in a family of maps over a punctured disk, and from this one might deduce many examples of lower bounds on Lyapunov exponents in one-parameter families which become arbitrarily large as one approaches a boundary point. 

Before we continue,  note that  by~\cite[Theorem 3.2]{bj}) we have an equality,
\begin{equation}\label{eq:bj}L(f)-L(f|_H)=\log d + \int Gd\mu_C,\end{equation}
where $G$ is the Green's function for $f$,\ and \[\mu_C=\frac{1}{(2\pi)^N}dd^c\log|\det Df|\wedge (dd^cG)^{N-1}\] is the critical measure of $f$. The aforementioned bound of Bedford and Jonsson follows from the non-negativity of the integral. The locus of pairs $(A, \mathbf{b})$ where the integral in~\eqref{eq:bj} vanishes is a natural generalization of the Mandelbrot set, and contains the image of $\SL_N(\CC)$ by $A\mapsto (A, \mathbf{0})$. The following corollary, then, gives a sort of generalization of the compactness of the Mandelbrot set.
\begin{corollary}\label{cor:proper}
	Let $\mathcal{M}\subseteq\SL_N(\CC)\times \CC^N$ be the set of pairs $(A, \mathbf{b})$ for which 
	\[L(f)-L(f|_H)-\log d = \int Gd\mu_C=0,\]
	for $f(\mathbf{X})=A\mathbf{X}^d+\mathbf{b}$. Then the projection $\mathcal{M}\to\SL_N(\CC)$ is proper.
\end{corollary}

In the arithmetic context, the \emph{critical height} of a morphism $f:\PP^N\to\PP^N$, denoted $\hat{h}_{\mathrm{crit}}(f)$, is the appropriate analogue of the sum of Lyapunov exponents. Extrapolating from a definition of Silverman~\cite[p.~101]{barbados}, we defined in \cite{pnheights} a critical height $\hat{h}_{\mathrm{crit}}(f)$ for an endomorphism $f:\PP^N\to\PP^N$ defined over $\overline{\QQ}$, with the property that PCF maps all have $\hat{h}_{\mathrm{crit}}(f)=0$. Specifically, we set
\[\hat{h}_{\mathrm{crit}}(f)=\hat{h}_f(C_f),\]
where $C_f$ is the critical locus of $f$, and $\hat{h}_f$ is the canonical height function associated to $f$ (constructed for subvarieties by Zhang~\cite{zhang}, but note that our canonical height for divisors here and in \cite{pnheights, mincrit} is Zhang's height times the degree of the divisor). Just as Silverman conjectured~\cite[p.~101]{barbados} in dimension 1 that the moduli height is an ample Weil height away from the Latt\`{e}s maps, confirmed in~\cite{hcrit}, it is natural to conjecture~\cite{pnheights} that $\hat{h}_{\mathrm{crit}}$ is an ample Weil height away from some proper, Zariski closed subset of moduli space. Theorem~\ref{th:main}, the arithmetic analogue of Theorem~\ref{th:proper}, proves this conjecture for fibres of the family~\eqref{eq:form} over $\SL_N$, with some uniformity as the fibre varies.

\begin{theorem}\label{th:main} 
	For $f:\PP^N\to\PP^N$ of the form~\eqref{eq:form} defined over $\overline{\QQ}$ (with $d\geq 2$), we have explicit constants $C_1$ and $C_2$, depending just on $N$ and $d$, such that
	\[ \hat{h}_{\mathrm{crit}}(f)- \hat{h}_{\mathrm{crit}}(f|_H)\geq\frac{d-1}{d}h(\mathbf{b})-\frac{N(dN+1)-1}{Nd}h(A)-C_1\]
	and
	\[ \hat{h}_{\mathrm{crit}}(f)- \hat{h}_{\mathrm{crit}}(f|_H)\leq N(N+2)h(\mathbf{b})+N(N+1)h(A)+C_2.\]
	In particular, the critical height is a moduli height for algebraic families in which $A$ is fixed. \end{theorem}

Note that the error term in the lower bound comes from a slightly better error term involving both $h(A)$ and $h(A^{-1})$.

For our next statement, recall that an endomorphism $f:\PP^N\to\PP^N$ is \emph{post-critically finite} (PCF) if and only if the post-critical locus
\[P_f=\bigcup_{n\geq 1}f^n(C_f)\]
is algebraic, where $C_f$ is again defined by the vanishing of the determinant of the Jacobian of $f$.

\begin{corollary}\label{cor:height}
Fix $B\geq 0$. For $A\in\SL_N(\overline{\QQ})$ with $h(A)\leq B$, the set of $\mathbf{b}\in\overline{\QQ}^N$ for which~\eqref{eq:form} is PCF is a set of bounded height, with bound depending just on $d, N$, and $B$.
\end{corollary}

 When $N=1$, McMullen~\cite[Theorem~2.2]{mcmullen} (building on work of Thurston) showed that the only non-isotrivial families of PCF rational functions are the flexible Latt\`{e}s examples. The next result gives a statement in this direction for families of the form~\eqref{eq:form}.

\begin{theorem}\label{th:rigid}
	Let $k$ be an algebraically closed field of characteristic $0$ or $p>d$. There is no algebraic family over $k$  	
	of PCF maps of the form~\eqref{eq:form} with $A$ constant, but $\mathbf{b}$ non-constant.
\end{theorem}

Compare with~\cite[Theorem~3]{mincrit}, which proves a similar result for a broader class of maps, but with restrictions on the degree. There are also cases in which we know that the induced family $f|_H(\mathbf{X})=A\mathbf{X}^d$, if PCF, must be constant (or at least isotrivial), and in those cases we get more out of Theorem~\ref{th:rigid}.
\begin{corollary}  On $\PP^N_\CC$ with $N=2$ or $d\geq N^2-N+1$, any algebraic family of PCF maps  of the form~\eqref{eq:form} is isotrivial.	
\end{corollary}

In the same cases, Corollary~\ref{cor:height} can also be improved.

	\begin{corollary}
	For fixed $d\geq N^2-N+1$ (or $d\geq 2$ if $N=2$), PCF maps $f:\PP^N\to\PP^N$ of the form~\eqref{eq:form} have conjugacy representatives contained in a set of bounded height.
	\end{corollary}

As alluded to above, this discussion fits into a larger framework. Let $\mathsf{M}_d^N$ be the moduli space of endomorphisms of $\PP^N$, and let $\mathsf{P}_d^N\subseteq \mathsf{M}_d^N$ be the space of regular polynomial endomorphisms, that is, elements of $\mathsf{M}_d^N$ with an invariant hyperplane. Restriction to the hyperplane gives a surjective morphism $\pi:\mathsf{P}_d^N\to \mathsf{M}_d^{N-1}$. Given that one always deserves a better understanding of $\mathsf{M}_d^{N-1}$ than one has of $\mathsf{M}_d^N$, it makes sense to approach $\mathsf{P}_d^N\subseteq \mathsf{M}_d^N$ by looking at what happens in fibres of the restriction map, and then thinking about how that relative behaviour varies as we vary the fibre.

In relation to the discussion of Silverman's Conjecture in~\cite{hcrit} and~\cite{pnheights}, note that on the relative moduli space $\mathsf{P}_d^N\to \mathsf{M}_d^{N-1}$ of regular polynomial endomorphisms, the function \[f\mapsto\hat{h}_{\mathrm{relcrit}}(f):= \hat{h}_{\mathrm{crit}}(f)-\hat{h}_{\mathrm{crit}}(f|_H)= \hat{h}_{\mathrm{crit}}(f)-\hat{h}_{\mathrm{crit}}(\pi(f))\] gives a non-negative function interacting nicely with iteration, \[\hat{h}_{\mathrm{relcrit}}(f^n)=n\hat{h}_{\mathrm{relcrit}}(f),\]
and vanishing precisely on those maps whose critical orbits are in some sense no more complex than they need be, given the behaviour at infinity. We offer a conjecture on this relative critical height which, while weaker than Silverman's, is perhaps also more approachable. A sufficiently precise version of this conjecture, combined with a version of Silverman's Conjecture in dimension one lower, ought to allow one to conclude Silverman's Conjecture for regular polynomial endomorphisms.

\begin{conjecture}[The relative critical height is a relative moduli height]\label{conj:relativehcrit}
For any ample Weil heights $h_{\mathsf{P}_d^N}$ and $h_{\mathsf{M}_d^{N-1}}$ on $\mathsf{P}_d^N$ and $\mathsf{M}_d^{N-1}$, we have
\[\hat{h}_{\mathrm{relcrit}}(f)\asymp  h_{\mathsf{P}_d^N}(f)+O(h_{\mathsf{M}_d^{N-1}}(f|_H)).\]
\end{conjecture}
Theorem~\ref{th:main} gives a result in this direction for morphisms of a certain form.
 Note also that a case of Conjecture~\ref{conj:relativehcrit} appears to follow from the results in~\cite{pcfpn}, but unfortunately with an incompatible definition of the critical height. It remains to be seen if those results can be translated into the terms of the present article.

We conclude with a few questions about the relative moduli space $\pi:\mathsf{P}_d^N\to \mathsf{M}_d^{N-1}$, generalizing the results above.

\begin{question}\label{conj:proper}
Let $\pi:\mathsf{P}_d^N\to \mathsf{M}_d^{N-1}$ be relative moduli space of regular polynomial endomorphisms, and let $\mathcal{M}\subseteq \mathsf{P}_d^N$ be the locus where $L(f)=L(f|_H)+\log d$. Is  $\pi:\mathcal{M}\to \mathsf{M}_d^{N-1}$  proper?\end{question}

An affirmative answer to Question~\ref{conj:proper} would immediately answer the complex case of the following question on PCF maps (which Theorem~\ref{th:rigid} answers negatively for maps of the form~\eqref{eq:form}).

\begin{question}[Relative rigidity]\label{conj:relrig} Let $k$ be an algebraically closed field of characteristic 0 or $p>d$. Can there be an algebraic curve in the 
 PCF locus in $\mathsf{P}_d^{N}$ over $k$ which is contained in a fibre of the projection $\pi:\mathsf{P}_d^{N}\to \mathsf{M}_d^{N-1}$?\end{question}

Of course, if $\Gamma\subseteq \mathsf{P}_d^{N}$ is an algebraic curve in the PCF locus, not contained in a fibre of $\pi$, then $\pi(\Gamma)\subseteq \mathsf{M}_d^{N-1}$ is an algebraic curve of PCF maps. In the case $N=2$, then, a negative answer to Question~\ref{conj:relrig} would imply that any nontrivial algebraic one-parameter family of PCF regular polynomial endomorphisms of $\PP^2$ restricts to the line at infinity as a flexible Latt\`{e}s family. The apparent rarity of PCF maps in several variables (see, e.g., \cite{pcfzar}), in fact, makes it likely that a stronger statement is true.

It is somewhat illustrative to consider Question~\ref{conj:relrig} in the case $N=1$. Note that $\PP^0$ is a single point, and over an algebraically closed field so is $\mathsf{M}_d^0$ for each $d\geq 2$. In particular, any curve in the PCF locus of $\mathsf{P}_d^1$ is contained in the unique fibre of the map $\mathsf{P}_d^{1}\to \mathsf{M}_d^{0}$, and so the expected negative answer to Question~\ref{conj:relrig}  simply asserts that there are no non-isotrivial families of PCF polynomials in one variable.

Similarly, the unique (up to conjugacy) endomorphism of $\PP^0$ of degree $d$ has critical height 0, and so Conjecture~\ref{conj:relativehcrit} above asserts that the critical height is a moduli height for polynomials of one variable, which is also true~\cite{pcfpoly}.

Before proceeding, we briefly mention the connection between this paper and~\cite{mincrit}. In~\cite{mincrit} we show that the critical height is a moduli height for maps of the form $f(\mathbf{X})=A\mathbf{X}^d$ in projective coordinates (one $d$ is sufficiently large compared to the dimension), and in the present paper we consider the subclass of these maps fixing one of the ramified hyperplanes. The main novelty in this special case is that we are able to conclude local results (i.e., over $\CC$ and $\CC_p$) which eluded us in~\cite{mincrit}. At the same time, the results of Theorem~\ref{th:main} and~\ref{th:rigid} have the benefit of applying to all degrees, but the drawback of depending on the behaviour at infinity, while the results in~\cite{mincrit} were absolute. These relative results can of course be combined with the results of~\cite{mincrit} applied to the map restricted to the invariant hyperplane, and we have demonstrated that in various places. In general, the estimate needed to deduce the results in this note are somewhat more delicate than those in~\cite{mincrit}, and have at least the potential to be extended to regular polynomial endomorphisms in general.

In Section~\ref{sec:pushpull} we work over an algebraically closed field, complete with respect to some absolute value, and prove most of the technical lemmas. Section~\ref{sec:relesc} introduces a ``relative rate of escape'' for a hypersurface under a map of the form~\eqref{eq:form}, which we then use to prove Theorem~\ref{th:proper} and Corollary~\ref{cor:proper}, as well as a statement of good reduction. Section~\ref{sec:global} contains the proofs of  Theorem~\ref{th:main} and Corollary~\ref{cor:height}. Finally, in Section~\ref{sec:p2} we delve deeper into the case $N=2$, boot-strapping some results from what is known about critical dynamics in one variable.

\section{Estimates on pulling-back and pushing-forward}\label{sec:pushpull}

Let $K$ be an algebraically closed field, complete with respect to some absolute value $|\cdot|$. We write $\|x_1, ... ,x_n\|=\max\{|x_1|, ..., |x_n|\}$, and $\log^+ x = \max\{\log x, 0\}$. Note that the triangle and ultrametric inequalities combine to give the following estimate, of which we make liberal use:
\[\log|x_1+\cdots +x_n|\leq \log\|x_1, ..., x_n\|+\log^+|n|.\]
Note that an absolute value is non-archimedean precisely if $\log^+|n|=0$ for all $n\in \ZZ$. To avoid unnecessary case distinctions in several places,  we adopt the conventions that $\log 0 = -\infty$, and that $-\infty < x< \infty$  and $\infty + x = \infty$ for all real numbers $x$.

Before continuing, we comment on various numbered constants $c_i$ that appear in the arguments below. Throughout, these constants have always been chosen to be non-negative, to simplify the manipulation of inequalities (sometimes at the cost of optimal bounds). Moreover, the constants $c_i$ depend only on $d$ (the degree of the endomorphisms $f$ under consideration) and $N$ (the dimension of the ambient space). Finally, these constants will all have value 0, except in the case of archimedean absolute values, or $p$-adic absolute values for $p\leq d$.

Given a homogeneous form $F$ with coefficients in $K$, we set $\|F\|$  to be the largest absolute value of a coefficient of $F$, in other words the  Gau\ss\ norm when $|\cdot|$ is non-archimedean. In~\cite{pnheights} we used the Mahler measure at the archimedean places, which is more natural, but turns out to be less convenient for the estimates in this note, which follows~\cite{mincrit} closely.

\begin{lemma}\label{lem:mult}
For $1\leq i\leq n$, let $F_i$ be a homogeneous form in $N+1$ variables. Then
\[-2N\sum_{i=1}^n\deg(F_i)\log^+|2|\leq \log\left\|\prod_{i=1}^{n}F_i\right\|-\sum_{i=1}^{n}\log\|F_{i}\|\leq 2N\sum_{i=1}^n\deg(F_i)\log^+|2|.\]

Let $F_{i, j}$ be homogeneous forms in $N+1$ variables such that, for each $1\leq i\leq n$, the form $\prod_{j=1}^{m_i}F_{i, j}$ has degree $\delta$. Then
\begin{equation}\label{eq:sumprod}\log\left\|\sum_{i=1}^n \prod_j^{m_i} F_{i, j}\right\|\leq\max_{1\leq i\leq n}\sum_{j=1}^{m_i}\log\|F_{i, j}\|+\log^+|n|+ 2N\delta \log^+|2|.\end{equation}
\end{lemma}

\begin{proof}
If $|\cdot|$ is non-archimedean, then these claims follow from the Gau\ss\ Lemma and the ultrametric inequality (and note that the error terms containing $\log^+|m|$, for $m$ an integer, vanish).

In the archimedean case, we recall the logarithmic Mahler measure of $F$, defined as
\[m(F)=\int\log|F|d\mu,\]
where $\mu$ is the usual normalized Haar measure on the unit circle in each variable. On the one hand, it is clear from the definition that $m(FG)=m(F)+m(G)$. On the other, it turns out that the Mahler measure is not too different from $\log\|F\|$, as pointed out by Mahler~\cite{mahler}. Specifically,
\begin{equation}\label{eq:mahler}m(F)-\frac{N}{2}\log(\deg(F)+1)\leq \log\|F\|\leq m(F)+N\deg(F)\log 2.\end{equation}
Notice
\begin{eqnarray*}
	\log\left\|\prod_{i=1}^{n}F_i\right\|&\leq & m\left(\prod_{i=1}^{n}F_i\right)+N\deg\left(\prod_{i=1}^nF_i\right)\log 2\\
	&=&\sum_{i=1}^nm\left(F_i\right)+N\sum_{i=1}^n\deg\left(F_i\right)\log 2\\
	&\leq & \sum_{i=1}^n \log\|F_i\|+N\sum_{i=1}^n\left(\frac{1}{2}\log(\deg(F_i)+1)+\deg\left(F_i\right)\log 2\right)\\
		&\leq & \sum_{i=1}^n \log\|F_i\|+2N\log 2\left(\sum_{i=1}^n\deg(F_i)\right)
\end{eqnarray*}
using the estimates $\log(1+x)\leq x$, for $x\geq 0$, and $\frac{1}{2}\leq \log 2$.

The inequality in the other direction is derived similarly.


For~\eqref{eq:sumprod}, by the triangle inequality,
\begin{eqnarray*}
	\log\left\|\sum_{i=1}^n \prod_j^{m_i} F_{i, j}\right\|&\leq & \max_{1\leq i\leq n}\log\left\|\prod_j^{m_i} F_{i, j}\right\|+\log n\\
	&\leq & \max_{1\leq i\leq n}m\left(\prod_j^{m_i} F_{i, j}\right)+N\delta\log 2+\log n\\
	&=&\max_{1\leq i\leq n}\sum_j^{m_i}m\left( F_{i, j}\right)+N\delta\log 2+\log n\\
	&\leq & \max_{1\leq i\leq n}\sum_j^{m_i}\log\left\| F_{i, j}\right\|+\max_{1\leq i\leq n}\sum_j^{m_i}\frac{N}{2}\log(\deg(F_{i, j})+1)\\&&+N\delta\log 2+\log n,
\end{eqnarray*}
which gives the desired bound again using $\log(1+x)\leq x$ and $\frac{1}{2}\leq \log 2$.
\end{proof}

Let $H$ denote the hyperplane of $\PP^N$ defined by $X_{N+1}=0$, and for any effective divisor $D$ on $\PP^N$ intersecting $H$ properly, and defined by $F=0$,  set
\[\lambda(D)=\log\|F(X_1, ..., X_{N+1})\|-\log\|F(X_1, ..., X_N, 0)\|.\]
This definition does not depend on the choice of homogeneous form $F$ representing $D$, and  $\lambda(D)\geq 0$. For some intuition, observe that that on $\PP^1$, we have
\[\lambda([z])=\log^+|z|\]
for the divisor $[z]$ corresponding to the point $z\in \PP^1_K\setminus\{\infty\}$.
We will also define, for a divisor $D$ defined by the homogenous form $F(\mathbf{X})=\sum_{i=0}^{\deg(F)} X_{N+1}^kF_k(X_1, ..., X_N)$ the quantity
\[\mu(D)=\min_{0\leq k<\deg(F)}\frac{\log|F_{\deg(D)}|-\log\|F_k\|}{\deg(D)-k},\]
provided that $D$ does not contain $H$ or $(0, 0, ..., 1)$, in which case we have $F_{\deg(D)}=F(0, 0, ..., 1)\neq 0$.
Note that it follows immediately from the definitions that
\[\mu(D)\leq \frac{\lambda(D)}{\deg(D)},\]
but there is no bound in the other direction.

\begin{remark}  
Although it is most efficient and transparent here to work in terms of homogeneous forms, it is worth noting that what we are doing fits into the framework of the geometry of arithmetic varieties as studied in arithmetic intersection theory. More concretely, if the absolute value on $K$ is non-archimedean, then $K$ has a ring of integers $\mathcal{O}\subseteq K$, and the morphism $f:\PP^N_K\to\PP^N_K$ extends to a rational map of schemes $f:\PP^N_\mathcal{O}\dashrightarrow 	\PP^N_\mathcal{O}$. The homogeneous  form $F(\mathbf{X})\in \mathcal{O}[X_1, ..., X_{N+1}]$ now defines an effective divisor $\PP^N_\mathcal{O}$, specifically $\overline{D}-\log_v\|F\|\PP_{k}^N$, where $\overline{D}$ is the Zariski closure of the divisor defined by $F$ on the generic fibre, $\PP_{k}^N$ is the special fibre, and $\log_v$ is normalized so that $\log_v|\pi|=-1$ for any uniformizer $\pi$ of the maximal ideal of $\mathcal{O}$. Our estimates on how $\log\|F\|$ changes under pulling-back by (some model of) $f$ now correspond to estimates on the difference between $f^*\overline{D}$ and $\overline{f^* D}$, for divisors $D$ on the generic fibre (but there appears to be no simpler way of making these estimates than to reduce things to computations involving homogeneous forms).
 Our estimates on pushing-forward are somewhat more fraught in this context, since $f:\PP^N_\mathcal{O}\dashrightarrow 	\PP^N_\mathcal{O}$ is generally not a morphism, but proceeding as in~\cite{pnheights} we may work with integral models, and recover something similar. All of these subtleties are eliminated by taking this more elementary approach.
\end{remark}

\begin{lemma}\label{lem:mulambda}
For effective divisors $D_i$, $1\leq i\leq n$, not containing $H$ we have\begin{multline}\label{eq:lambdaineq}
-4N\sum_{i=1}^n\deg(D_i)\log^+|2|\\\leq \lambda\left(\sum_{i=1}^n D_i\right)- \sum_{i=1}^n\lambda(D_i)\\\leq 4N\sum_{i=1}^n\deg(D_i)\log^+|2|,	
\end{multline}
and if the $D_i$ do not contain the origin, we also have
\begin{equation}\label{eq:muineq}\mu\left(\sum_{i=1}^n D_i\right)\geq\min_{1\leq i\leq n}\mu(D_i)-2N\log^+|2|-(n-1)\log^+\left|\sum_{i=1}^n\deg(D_i)\right|.\end{equation}
Finally, if $\mu(D)\geq 0$, then
\[\lambda(D)=\log|F_{\deg(D)}|-\log\|F_0\|\]
for any form $F(\mathbf{X})=\sum_{i=0}^{\deg(F)} X_{N+1}^iF_i(X_1, ..., X_N)$ defining $D$.
\end{lemma}

\begin{proof}
The claim~\eqref{eq:lambdaineq} follows immediately from Lemma~\ref{lem:mult}.

For~\eqref{eq:muineq} in the non-archimedean case, the proof is similar to that of the Gau\ss\ Lemma. 

Specifically, let $D_i$ be defined by $F_i=0$, with
\[F_i(\mathbf{X})=\sum_{j=0}^{\deg(F_i)}X_{N+1}^jF_{i, j}(X_1, ..., X_N),\]
 and choose $k_i$ minimally so that
\[\mu(D_i)=\frac{\log|F_{i, \deg(F_i)}|-\log\|F_{i, k_i}\|}{\deg(F_i)-k_i}.\]
Now, for $\delta=\sum_{i=1}^n\deg(D_i)$ we have $\sum_{i=1}^n D_i$  defined by the vanishing of
\[\sum_{\ell = 0}^{\delta}X_{N+1}^\ell G_{\ell}\] where \[G_\ell= \sum_{j_1+\cdots +j_n=\ell}\prod_{i=1}^nF_{i, j_i}.\]
Note that the number of summands is the number of solutions to $j_1+\cdots + j_n=\ell$ satisfying $0\leq j_i\leq \deg(F_i)$ for all $0\leq i \leq n$, which we crudely estimate as at most $(\ell+1)^{n-1}$.
 Then we have by Lemma~\ref{lem:mult}
\begin{multline*}
\log\left\|G_\ell\right\|	\leq\max_{j_1+\cdots +j_n=\ell}\sum_{i=1}^{n}\left(\log\|F_{i, j_i}\|+2N\left(\deg(D_i)-j_i\right)\log^+|2|\right)\\+(n-1)\log^+|\ell+1|.	
\end{multline*}
Choosing $j_1+\cdots+j_n=\ell$ maximizing the right-hand-side, we have
\begin{eqnarray*}
\log\left|G_{\delta}\right|-\log\left\|G_\ell\right\|&\geq & \sum_{i=1}^n\left(\log|F_{i, \deg(F_i)}|-\log\|F_{i, j_i}\|\right)\\&&-2N\sum_{i=1}^n\left(\deg(D_i)-j_i\right)\log^+|2| - (n-1)\log^+|\ell+1|\\
&\geq & \sum_{i=1}^n\left(\mu(D_i)-2N\log^+|2|\right)(\deg(D_i)-j_i)\\&&-(n-1)\log^+|\ell+1|\\
&\geq &\left(\min_{1\leq i\leq n}\mu(D_i)-2N\log^+|2|\right)\sum_{i=1}^n(\deg(D_i)-j_i)\\&&-(n-1)\log^+|\ell+1|\\
&=&\left(\min_{1\leq i\leq n}\mu(D_i)-2N\log^+|2|\right)\left(\delta-\ell\right)\\&&-(n-1)\log^+|\ell+1|.
\end{eqnarray*}
Dividing both sides by $\delta-\ell$ and taking the minimum over $0\leq \ell<\delta$ gives
\[\mu\left(\sum_{i=1}^n D_i\right)\geq \min_{1\leq i\leq n}\mu(D_i)-2N\log^+|2|-(n-1)\max_{0\leq \ell < \delta}\left(\frac{\log^+|\ell+1|}{\delta-\ell}\right).\]
Note that the maximum in the last term is attained at $\ell=\delta-1$.

The last claim is simply due to the fact that $\mu(D)\geq 0$ implies \[\log|F_{\deg(D)}|\geq \log\|F_k\|\] for all $0\leq k<\deg(D)$, whence $\log\|F\|=\log|F_{\deg(D)}|$.
\end{proof}

We will fix a  block matrix
\begin{equation}\label{eq:block}L=\begin{pmatrix}A & \mathbf{b} \\ \mathbf{0}  & 1	 \end{pmatrix},\end{equation}
where $A$ is $N\times N$, $\mathbf{b}$ is $N\times 1$, and $\mathbf{0}$ is the $1\times N$ zero vector,
and use the same symbol to denote the resulting linear map $L:\PP^{N}\to\PP^{N}$.
Note that the inverse map/matrix is given by
\[L^{-1}=\begin{pmatrix}A^{-1} & -A^{-1}\mathbf{b} \\ \mathbf{0}  & 1	 \end{pmatrix}.\]

We will also write $\phi$ for the power map of degree $d$ on $\PP^{N}$, so that
\[\phi(X_1, ..., X_{N+1})=[X_1^d:\cdots : X_{N+1}^d].\]
Note that endomorphisms of $\PP^N$ of the form under consideration, described in~\eqref{eq:form}, are precisely those of the form $f=L\circ\phi$.

We will be interested in the behaviour of the quantities $\lambda$ and $\mu$ under pushing-forward and pulling-back divisors by $f$, and so consequently by $L$ and by $\phi$.
First, the power map.

\begin{lemma}\label{lem:pushpullpower}
For $\phi$ as above, and any effective divisor $D$,
\[\lambda(\phi^* D)=\lambda(D)\quad\text{ and }\quad\mu(\phi^* D) =\frac{1}{d}\mu(D),\]
\[
\left|\lambda(\phi_* D)-d^{N}\lambda(D)\right|\leq 4Nd^{N}\deg(D)\log^+|2|,
\]
and
\[\quad\mu(\phi_* D) \geq d\mu(D)-2dN\log^+|2|-d(d^N-1)\log^+\left|d^N\deg(D)\right|.\]	
\end{lemma}

\begin{proof}
Let $F$ be some homogeneous form whose vanishing defines $D$. For the pull-back, notice that $\phi^* D$ is defined by $F(X_1^d, ..., X_{N+1}^d)$. The coefficients of this homogeneous form are exactly those of $F$ (associated to different monomials), and so we certainly have $\|F\|=\|F(X_1^d, ..., X_{N+1}^d)\|$. The claim about $\lambda(\phi^*D)$ follows immediately, while the claim about $\mu(\phi^* D)$ follows once we note that $\phi^* D$ has degree $d\deg(D)$.

Now consider the push-forward. For any tuple $\mathbf{\zeta}=(\zeta_1, ..., \zeta_N)$ of $d$th roots of unity, let $D_{\mathbf{\zeta}}$ be the divisor defined by the vanishing of $F_{\mathbf{\zeta}}(X_1, ..., X_{N+1})=F(\zeta_1X_1, ..., \zeta_NX_N, X_{N+1})$, noting that 
\[\phi^*\phi_* D=\sum_{\zeta_1^d=\cdots =\zeta_{N}^d=1}D_\mathbf{\zeta}.\]
Also, note that $\lambda(D_\mathbf{\zeta})=\lambda(D)$ and $\mu(D_\mathbf{\zeta})=\mu(D)$, since the coefficients of $F_\mathbf{\zeta}$ are the coefficients of $F$ multiplied by various roots of unity.
By~\eqref{eq:lambdaineq} of Lemma~\ref{lem:mulambda} and the estimates for the pullback above, we have
\begin{eqnarray*}
\lambda(\phi_*D)&=&\lambda(\phi^*\phi_*D)\\
&=&\lambda\left(\sum_{\zeta_1^d=\cdots =\zeta_{N}^d=1}D_\mathbf{\zeta}\right)\\
&\leq&\sum_{\zeta_1^d=\cdots =\zeta_{N}^d=1}\lambda(D_\mathbf{\zeta})+4N\sum_{\zeta_1^d=\cdots =\zeta_{N}^d=1}\deg(D_\mathbf{\zeta})\log^+|2|\\
&=&d^N\lambda(D)+4Nd^{N}\deg(D)\log^+|2|	
\end{eqnarray*}
and, by the essentially the same calculation,
\[\lambda(\phi_*D)\geq d^N\lambda(D)-4Nd^{N}\deg(D)\log^+|2|.\]
Meanwhile,
\begin{eqnarray*}
\mu(\phi_* D)&=&d\mu(\phi^*\phi_* D)\\
&=&d\mu\left(\sum_{\zeta_1^d=\cdots =\zeta_{N}^d=1}D_\mathbf{\zeta}\right)\\
&\geq &	d\min\{\mu(D_\mathbf{\zeta})\}-2dN\log^+|2|-d(d^N-1)\log^+\left|\sum_{\zeta_1^d=\cdots =\zeta_{N}^d=1}\deg(D_\mathbf{\zeta})\right|\\
&=&d\mu(D)-2dN\log^+|2|-d(d^N-1)\log^+\left|d^N\deg(D)\right|.
\end{eqnarray*}
\end{proof}

Next we will estimate $\lambda(L^*D)$ and $\lambda(L_*D)$, and $\mu(L^*D)$ and $\mu(L_*D)$. But since our error terms will depend on the matrices representing these linear maps, it makes sense to introduce some N\'{e}ron functions  on matrices. The following lemma is easy to check, and left to the reader.
\begin{lemma} For a matrix $A$ with $i, j$th entry $A_{i, j}$, let $\|A\|=\max_{i, j}|A_{i, j}|$.
	Let $\lambda:\SL_{N}(K)\to \RR$ be defined by
	\[\lambda(A)=N\log\|A\|+\log^+|N!|,\]
	and $\xi:\SL_{N}(K)\to \RR$ by
	\[\xi(A)=\log\|A\|+\log\|A^{-1}\|+\log^+|N|.\]
	Then the functions $\lambda$ and $\xi$  are non-negative, and satisfy
 \begin{equation}\label{eq:Ainv}\xi(A)\leq \lambda(A)+\log^+|N|\end{equation}
 and
 \begin{equation}\label{eq:lambdainv}\lambda(A^{-1})\leq (N-1)\lambda(A)\end{equation}
 
\end{lemma}

Lemma~\ref{lem:pushpullpower} describes the behaviour certain quantities associated to divisors under pushing-forward or pulling-back by the power map, and now we present a corresponding result relative to linear maps. These estimates are very similar to those in the proof of~\cite[Lemma~10]{mincrit}, but the precise bounds depend on the special form of the matrix $L$.
\begin{lemma}\label{lem:pushpulllinear}
 For  
 \[c_1=(2N-1)\log^+|2|\qquad\text{and}\qquad c_2=\log^+|4N(N+1)|,\]
 $L$ as in~\eqref{eq:block},
 and  any effective divisor $D$, we have
 \begin{multline*}
-\deg(D)\Big(\log\|L\|-\log\|A\| 
+\lambda(L)+c_2\Big) - c_1 \\\leq \lambda(L^* D)-\lambda(D) \\\leq \deg(D)\Big(\log\|L\|-\log\|A\| 
+\lambda(A)+c_2\Big) + c_1
 \end{multline*}
and
\begin{multline*}
-\deg(D)\Big(\log\|L\|-\log\|A\| 
+\lambda(A)+c_2\Big) - c_1 \\\leq \lambda(L_* D)-\lambda(D) \\\leq \deg(D)\Big(\log\|L\|-\log\|A\| 
+\lambda(L)+c_2\Big) + c_1.
\end{multline*}
\end{lemma}

\begin{proof} Choose a defining homogeneous form $F$ for $D$.
Any homogeneous form $F(\mathbf{X})=\sum c_mm(\mathbf{X})$ is a linear combination of at most $\binom{\deg(F)+N}{N}$ monomials of degree $\deg(F)$, and so we have for any $B\in \SL_{N+1}(K)$
\begin{eqnarray}\log\|F(B\mathbf{X})\|&=& \log\left\|\sum c_m m(B\mathbf{X})\right\|\nonumber \\
&\leq&\log\max \|c_m m(B\mathbf{X})\| + \log^+\left|\binom{\deg(F)+N}{N}\right| \nonumber \\
&\leq&\log\|F\|+\deg(F)\log\|B\|+\deg(F)\log^+|N+1|\label{eq:basicpullback} \\&&+\deg(F)\log^+|2|+N\log^+|2|\nonumber
	\end{eqnarray}
by the triangle inequality. On the other hand, by~\eqref{eq:Ainv}  we have
\begin{eqnarray*}
\log\|F\|&=&\log\|F(BB^{-1}\mathbf{X})\|\\
&\leq & \log\|F(B\mathbf{X})\|+\deg(F)\log\|B^{-1}\|+\deg(F)\log^+|N+1|\\&&+\deg(F)\log^+|2|+N\log^+|2|\\
&\leq & \log\|F(B\mathbf{X})\| +\deg(F)\lambda(B)-\deg(F)\log\|B\|\\
&& +\deg(F)\log^+|N+1|+\deg(F)\log^+|2|+N\log^+|2|.
\end{eqnarray*}

Now for $L$ of the form~\eqref{eq:block}, note that if $F|_H(X_1, ..., X_N)=F(X_1, ..., X_N, 0)$, then
$(F\circ L)|_H=(F|_H)\circ A$,
so for $D$ defined by $F=0$ we have 
\begin{eqnarray*}
\lambda(L^*D)&=&\log\|F\circ L\| - \log\|F_0\circ A\|\\
&\leq & \log\|F\|+\deg(F)\log\|L\|+\deg(F)\log^+|N+1| \\&&+\deg(F)\log^+|2|+N\log^+|2|\\&&
-\log\|F_0\| +\deg(F)\lambda(A)-\deg(F)\log\|A\|\\
&& +\deg(F)\log^+|N|+\deg(F)\log^+|2|+(N-1)\log^+|2|\\
&=&\lambda(D)+\deg(F)\Big(\log\|L\|-\log\|A\|+\log^+|4N(N+1)| 
+\lambda(A)\Big)\\&& + (2N-1)\log^+|2|
\end{eqnarray*}
Similarly,
\begin{eqnarray*}
\lambda(L^*D)&=&\log\|F\circ L\| - \log\|F_0\circ A\|\\
&\geq &\log\|F\|-\deg(F)\lambda(L)+\deg(F)\log\|L\|\\
&& -\deg(F)\log^+|N+1|-\deg(F)\log^+|2|-N\log^+|2|\\
&& - \log\|F_0\|-\deg(F)\log\|A\|-\deg(F)\log^+|N|\\&&-\deg(F)\log^+|2|-(N-1)\log^+|2|\\
&=&\lambda(D)-\deg(D)\Big(\lambda(L)+\log\|L\|-\log\|A\|+\log^+|4N(N+1)|\Big)\\&& - (2N-1)\log^+|2|.
\end{eqnarray*}

The bounds for $\lambda(L_*D)$ follow immediately from writing $D=L^ *L_*D$
\end{proof}

Lemma~\ref{lem:pushpulllinear} gives estimates on $\lambda(L_* D)-\lambda(D)$ which depend on $L$, as one might expect. However, by analogy with $z\mapsto z+c$, one might also expect much more uniform estimates once $D$ is sufficiently ``large'' with respect to the coefficients of $L$, estimates which depend only on the behaviour at infinity.
The rest of the section is more technical, and gives such estimates.

\begin{lemma}\label{lem:tc} 
For $\mathbf{c}=(c_1, ..., c_N)\in K^N$, let \[T_{\mathbf{c}}(X_1, ..., X_{N+1})=(X_1+c_1X_{N+1}, ..., X_N+c_NX_{N+1}, X_{N+1})\] be the translation-by-$\mathbf{c}$ map,
 let $D$ be a divisor not containing $H$ or the origin, and let
\[c_3= \begin{cases} (N+2)\log 2+\log N & \text{if $|\cdot|$ is archimedean,}\\\frac{\log p}{p-1} & \text{if $|\cdot|$ is $p$-adic,}\\
0 & \text{otherwise.}	
\end{cases}\]
If  \[\mu(D)>\log^+\|\mathbf{c}\|+c_3+2\log^+|\deg(D)|,\]
then $T_{\mathbf{c}}^*D$ also does not contain $H$ or the origin, and we have
\[\mu(T_{\mathbf{c}}^*D)\geq \mu(D)-\log^+|\deg(D)|-\log^+|2|.\] 
\end{lemma}

\begin{proof}
If $D$ is defined by the vanishing of
\[F(X_1, ..., X_{N+1})=F_0(X_1, ..., X_N)+X_{N+1}F_1(X_1, ..., X_N)+\cdots +X_{N+1}^{\deg(D)}F_{\deg(D)},\]
then $T_\mathbf{c}^*D$ is defined by the vanishing of
\[E(X_1, ..., X_{N+1})=F(X_1+c_1X_{N+1}, ..., X_N+c_NX_{N+1}, X_{N+1}),\]
which we would like to write as
\[E(X_1, ..., X_{N+1})=E_0(X_1, ..., X_N)+X_{N+1}E_1(X_1, ..., X_N)+\cdots +X_{N+1}^{\deg(D)}E_{\deg(D)}.\]

With a view to computing $\mu(T_\mathbf{c}^*D)$, note that
\[E_{\deg(D)}=E(0, 0, ..., 1)=F(c_1, ..., c_N, 1),\]
and so
\begin{eqnarray}
\nonumber\log |E_{\deg(D)}-F_{\deg(D)}|&= & \log|F(c_1, ..., c_N, 1)-F_{\deg(D)}|\\
\nonumber &=&\log\left|\sum_{k=1}^{\deg(D)} F_{\deg(D)-k}(\mathbf{c})\right|\\
\nonumber &\leq & \max_{1\leq k\leq \deg(D)}\Big(k\log\|\mathbf{c}\|+\log\|F_{\deg(D)-k}\|\\
\nonumber &&+\log^+\left|\binom{\deg(F_{\deg(F)-k})+N}{N}\right|\Big)+\log^+|\deg(D)|\\
\nonumber &\leq & \max_{1\leq k\leq \deg(D)}\Big(k\mu(D)-kc_3-k2\log^+|\deg(D)|\\
\nonumber &&+\log\|F_{\deg(D)-k}\|+k(N+1)\log^+|2|\Big)\\
\nonumber &&+\log^+|\deg(D)|\\
\nonumber & < & \log|F_{\deg(D)}|+\log^+|\deg(D)|\\
 &&-\min_{1\leq k\leq \deg(D)}k(c_3+2\log^+|\deg(D)|-(N+1)\log^+|2|)\label{eq:mink}\\
\nonumber &\leq & \log|F_{\deg(D)}|-\log^+|2|,\nonumber
\end{eqnarray}
since the minimum in~\eqref{eq:mink} is attained with $k=1$ (the term in parentheses being non-negative).
 So we get
\[\log|E_{\deg(D)}|\geq \log|F_{\deg(D)}|-\log^+|2|,\]
and also $E_{\deg(D)}\neq 0$, which is equivalent to $T^*_{\mathbf{c}}D$ not containing the origin.

In order to obtain a lower bound on $\mu(T^*_{\mathbf{c}})$, we now need an upper bound on $\|E_s\|$ for $s<\deg(D)$.
We can expand each $F_\ell\circ T_{\mathbf{c}}$ as a polynomial in $X_{N+1}$ in a fairly simple manner, namely by
\[F_\ell\circ T_{\mathbf{c}} (\mathbf{X})=\sum_{j=0}^{\deg(F_\ell)}\frac{X_{N+1}^j}{j!}\left(\frac{\partial^j (F_\ell\circ T_{\mathbf{c}})}{\partial X_{N+1}^j}\Big|_{X_{N+1}=0}\right).\]
By the chain rule, if we write
\[F_{i, k_1, ..., k_j}=\frac{\partial^jF_i}{\partial X_{k_1}\cdots \partial X_{k_j}},\]
then
\begin{equation}\label{eq:partiald}\frac{\partial^j (F_\ell\circ T_{\mathbf{c}})}{\partial X_{N+1}^j}(X_1, ..., X_N, 0)=\sum_{k_1, ..., k_j=1}^Nc_{k_1}\cdots c_{k_j}F_{\ell, k_1, ..., k_j}(X_1, ..., X_N).\end{equation}
For any homogeneous form $H$,
\[\log\left\|\frac{\partial^j H}{\partial X_{k_1}\cdots \partial X_{k_j}}\right\| \leq \log\|H\|+j\log^+|\deg(H)|,\]
and so each summand on the right-hand side of~\eqref{eq:partiald} satisfies
\[\log\|c_{k_1}\cdots c_{k_j}F_{\ell, k_1, ..., k_j}\|\leq j\log\|\mathbf{c}\|+\log\|F_{\ell}\|+j\log^+|\deg(D)-\ell|\]
Summing over all terms on the right in~\eqref{eq:partiald} then gives
\begin{multline*}
\log\left\|\frac{1}{j!}\frac{\partial^j (F_\ell\circ T_{\mathbf{c}})}{\partial X_{N+1}^j}(X_1, ..., X_N, 0)\right\|\\ \leq \log\|F_{\ell}\|+j\Big(\log\|\mathbf{c}\|+\log^+|\deg(D)-\ell|+\log^+|N|\Big)+\log^+\left|\frac{1}{j!}\right|	
\end{multline*}

At this point, we note that if $|\cdot|$ is not a $p$-adic absolute value, for any prime integer $p$, then $\log^+|\frac{1}{j!}|=0$. If $|\cdot|$ is the $p$-adic absolute value,
\[\log^+\left|\frac{1}{j!}\right|=\sum_{t=1}^\infty\left\lfloor \frac{j}{p^t}\right\rfloor\log p \leq j\left(\frac{\log p}{p-1}\right)\]
by Legendre's formula, and so either way
\[\log^+\left|\frac{1}{j!}\right|\leq jc_4,\]
where \[c_4=\begin{cases} \frac{\log p}{p-1} & \text{if $|\cdot|$ is $p$-adic,}\\
0 & \text{otherwise.}\end{cases}\]

Now, comparing coefficients of $X_{N+1}^s$, we have
\[E_s(X_1, ..., X_N)=\sum_{j=0}^s\frac{1}{j!}\frac{\partial^j F_{s-j}\circ T_{\mathbf{c}}}{\partial X_{N+1}^j}(X_1, ..., X_N, 0),\]
whence
\begin{eqnarray}
\log\|E_s\|&\leq& \max_{0\leq j\leq s}\Big\{\log\|F_{s-j}\|+j\Big(\log\|\mathbf{c}\|+\log^+|\deg(D)|+\log^+|N|+c_4\Big)\Big\}\nonumber\\&&+\log^+|s+1|\nonumber\\
  \nonumber\\&\leq & \max_{0\leq j\leq s}\Big\{\log|F_{\deg(D)}|-(\deg(D)-s+j)\mu(D)\nonumber\\&&+j\Big(\mu(D)-c_3-2\log^+|\deg(D)|+\log^+|\deg(D)|+\log^+|N|+c_4\Big)\Big\}\nonumber \\ &&+\log^+|\deg(D)|\nonumber\\ 
&\leq & \log|E_{\deg(D)}|+\log^+|2|-(\deg(D)-s)\mu(D)+\log^+|\deg(D)|\label{eq:es},
\end{eqnarray}
since
\[c_3\geq \log^+|2|+\log^+|N|+c_4\]
But~\eqref{eq:es} for all $0\leq s<\deg(D)$ gives $\mu(T^*_{\mathbf{c}}D)\geq \mu(D)-\log^+|\deg(D)|-\log^+|2|$.
\end{proof}

Lemma~\ref{lem:tc} effectively gives estimates on pushing-forward or pulling-back by $L$, in the special case where $A$ is the identity matrix. It turns out that, with a little more work, this special case gives us the general case.
\begin{lemma}\label{lem:keything}
Let 
\[c_5=\log^+|N|+N\log^+|2|+\frac{1}{N}\log^+|N!|\geq 0,\] 
let $D$ be an effective divisor of degree at least 1, not containing the origin, and suppose that 
  \begin{equation}\label{eq:mulower}\mu(D)>\log^+\|\mathbf{b}\|+c_3+c_5+\log\|A^{-1}\|+2\log^+|\deg(D)|.\end{equation}
   Then \begin{equation}\label{eq:mustar}\mu (L_*D)\geq \mu(D)-\log^+|\deg(D)|-\log^+|2|-c_5-\log\|A^{-1}\|\end{equation} and \begin{multline} \lambda(D)-\deg(D)(\log\|A^{-1}\|+\log^+|2N|)-N\log^+|2|\\ \leq \lambda(L_*D)\\ \leq\lambda(D)+\deg(D)(\log\|A\|+\log^+|2N|)+N\log^+|2| \end{multline}
\end{lemma}

\begin{proof}
Note that $L_*=(L^{-1})^*$, and that
\[L^{-1}=\begin{pmatrix}
A^{-1} & -A^{-1}\mathbf{b}\\ \mathbf{0} & 1	
\end{pmatrix}=\begin{pmatrix}
A^{-1} &\mathbf{0}\\ \mathbf{0} & 1	
\end{pmatrix}\begin{pmatrix}
I & -\mathbf{b}\\ \mathbf{0} & 1	
\end{pmatrix}=: L_0^{-1}T_{-\mathbf{b}}.
\]
So $L_*=(L^{-1})^*=T_{-\mathbf{b}}^*(L_{0}^{-1})^*$.

First, note that $\mu((L_0^{-1})^* D)$ can be estimated as follows. If $D$ is defined by the vanishing of $F=\sum_{i=0}^{\deg(F)}F_iX_{N+1}^i$, then $(L_0^{-1})^*D$ is defined by the vanishing of $\sum_{i=0}^{\deg(F)} X_{N+1}^iF_i\circ A^{-1}$. From~\eqref{eq:basicpullback}, we have for any $0\leq k<\deg(D)$
\begin{eqnarray*} 		
\log|F_{\deg(D)}|-\log\|F_k\circ A^{-1}\| & \geq & \log|F_{\deg(D)}|-\log\|F_k\|-\deg(F_k)\Big(\log\|A^{-1}\|\\&&+\log^+|N|+\log^+|2|\Big)-(N-1)\log^+|2|	\\
&\geq&(\deg(D)-k)\mu(D)-\deg(F_k)(\log\|A^{-1}\|\\&&+\log^+|N|+N\log^+|2|)\\
&\geq&(\deg(D)-k)(\mu(D)-c_5-\log\|A^{-1}\|),
\end{eqnarray*}
(noting that this is trivially true if $F_k=0$) whence \[\mu((L_0^{-1})^*D)\geq \mu(D)-c_5-\log\|A^{-1}\|.\]
Combined with~\eqref{eq:mulower}
this gives
\begin{equation*}
\mu((L_0^{-1})^*D)\geq \mu(D)-c_5-\log\|A^{-1}\|
\geq  \log^+\|\mathbf{b}\|+c_3+2\log^+|\deg(D)|,
\end{equation*}
and so by Lemma~\ref{lem:tc} we have
\begin{multline}\label{eq:super}\mu(L_*D)=\mu(T_{-\mathbf{b}}^*(L_0^{-1})^* D)\geq \mu((L_0^{-1})^* D)-\log^+|\deg(D)|-\log^+|2|\\\geq \mu(D)-\log^+|\deg(D)|-\log^+|2|-c_5-\log\|A^{-1}\|\geq 0\end{multline}
proving~\eqref{eq:mustar}.

Since $A^{-1}\in \SL_N(K)$ we have
\[c_5+\log\|A^{-1}\|\geq c_5-\frac{1}{N}\log^+|N!|\geq 0,\]
and so  by~\eqref{eq:mulower} we have $\mu(D)\geq 0$, hence $\lambda(D)=\log|F_{\deg(D)}|-\log\|F_0\|$.  Similarly, by the computations giving~\eqref{eq:super} we have
\begin{eqnarray*}
\lambda(L_* D)&= & \log|F_{\deg(D)}|-\log\|F_0\circ A^{-1}\|\\
&\geq & \log|F_{\deg(D)}| - \log\|F_0\| -\deg(D)\log\|A^{-1}\|- \deg(D)\log^+|2N|\\&&-N\log^+|2|\\
&=&\lambda(D)-\deg(D)(\log\|A^{-1}\|+\log^+|2N|)-N\log^+|2|.
\end{eqnarray*}
Also, 
\begin{eqnarray*}
\lambda(L_* D)&= & \log|F_{\deg(D)}|-\log\|F_0\circ A^{-1}\|\\
&\leq & \log|F_{\deg(D)}| - \log\|F_0\| +\deg(D)\log\|A\|+ \deg(D)\log^+|2N|\\&&+N\log^+|2|\\
&=&\lambda(D)+\deg(D)(\log\|A\|+\log^+|2N|)+N\log^+|2|.
\end{eqnarray*}

\end{proof}

\section{The relative rate of escape}\label{sec:relesc}

We continue in the context of the last section. That is, $K$ is an algebraically closed field, complete with respect to some absolute value $|\cdot|$.

Let $f$ be as in~\eqref{eq:form}, and let $D$ be an effective divisor not containing $H$. 
We set
\begin{equation}\label{eq:greens}\Delta_{f}(D)=\lim_{k\to\infty} \frac{\lambda\left(f^k_*D\right)}{d^{kN}},\end{equation}
whenever this limit exists, but we will prove that it always does (subject to the constraints above).

\begin{lemma}\label{lem:basicdelta}
	The limit in~\eqref{eq:greens} exists, is non-negative for effective divisors $D$, and we have
	\[\Delta_f(f_*D)=d^{N}\Delta_f(D),\]
	\[\Delta_f(f^*D)=\Delta_f(D),\]
	and
	\[\Delta_f(D+E)=\Delta_f(D)+\Delta_f(E),\]
	as well as
\begin{multline}\label{eq:deltaineq}-\frac{\deg(D)}{d-1}\Big(\log\|L\|-\log\|A\|+\log^+|4N(N+1)|\\
+\lambda(A)+4Nd\log^+|2|\Big)-\left(\frac{2N-1}{d^N-1}\right)\log^+|2|\\\leq \Delta_f(D)-\lambda(D)\\\leq \frac{\deg(D)}{d-1}\Big(\log\|L\|-\log\|A\|+\log^+|4N(N+1)|\\
+\lambda(L)+4Nd\log^+|2|\Big)+\left(\frac{2N-1}{d^N-1}\right)\log^+|2|.
 \end{multline}
	Furthermore, if $D$ is preperiodic for $f$, then $\Delta_f(D)=0$.
\end{lemma}

\begin{proof}
We will first show that the limit exists, for which we apply Lemmas~\ref{lem:pushpullpower} and~\ref{lem:pushpulllinear}. Specifically, 
\begin{eqnarray*}
\left|\frac{\lambda(f_*D)}{d^{N}}-\lambda(D)\right|&\leq & \left|\frac{\lambda(L_*\phi_*D)}{d^{N}}-\frac{\lambda(\phi_*D)}{d^N}\right|+\left|\frac{\lambda(\phi_*D)}{d^N}-\lambda(D)\right|\\
&\leq & d^{-N}\deg(\phi_*D)(\log\|L\|-\log\|A\|+\lambda(L)+\lambda(A)+c_2)\\&&+d^{-N}c_1+4N\deg(D)\log^+|2|\\
&=&\deg(D)\Big(d^{-1}\left(\log\|L\|-\log\|A\| 
+\lambda(L)+\lambda(A)+c_2\right)\\&&+4N\log^+|2|\Big)+d^{-N}c_1.
\end{eqnarray*}
Since $\deg(f_*D)=d^{N-1}\deg(D)$, a standard telescoping sum argument gives
\begin{eqnarray*}
\left|\frac{\lambda(f_*^kD)}{d^{Nk}}-\lambda(D)\right|	&\leq&\sum_{j=0}^{k-1}\left|\frac{\lambda(f_*^{j+1}D)}{d^{N(j+1)}}-\frac{\lambda(f_*^jD)}{d^{Nj}}\right|\\
&\leq & \sum_{j=0}^{k-1}d^{-Nj}\deg(f_*^{j}D)\Big(d^{-1}\Big(\log\|L\|-\log\|A\|\\&&+\lambda(L)
+\lambda(A)+c_2\Big)+4N\log^+|2|\Big)\\&& +c_1\sum_{j=0}^{k-1}d^{-N(j+1)}\\
&= & \left(\frac{1-d^{-k}}{1-d^{-1}}\right)\deg(D)\Big(d^{-1}(\log\|L\|-\log\|A\|+\lambda(L)\\&& 
+\lambda(A))+c_2\Big)\\&& +\left(\frac{1-d^{-Nk}}{d^N-1}\right)c_1.
\end{eqnarray*}
The difference $d^{-Nk}\lambda(f_*^kD)-\lambda(D)$ is thus the partial sum of an absolutely convergent series, and hence the limit in~\eqref{eq:greens} exists. The above calculation, with slightly more care to distinguish the terms in the upper and lower bounds, and with $k\to\infty$, now gives~\eqref{eq:deltaineq}.

For linearity, note that in the non-archimedean case $\lambda(D+E)=\lambda(D)+\lambda(E)$ by the Gau\ss\ lemma, and so $\Delta_f$ is linear as well. In the archimedean case, note that Lemma~\ref{lem:mulambda} gives
\begin{eqnarray*}
\Delta_f(D+E)&=&\lim_{k\to\infty}\frac{\lambda(f_*^k D+f_*^kE)}{d^{Nk}}\\
	&=&\lim_{k\to\infty}\frac{\lambda(f_*^k D)+\lambda(f_*^kE)+O(\deg(f_*^k D+f_*^kE))}{d^{Nk}}\\
	&=&\left(\lim_{k\to\infty}\frac{\lambda(f_*^k D)}{d^{Nk}}+\lim_{k\to\infty}\frac{\lambda(f_*^k E)}{d^{Nk}}+\lim_{k\to\infty}\frac{d^{(N-1)k}O(\deg(D+E))}{d^{Nk}}\right)\\
	&=&\Delta_f(D)+\Delta_f(E).
\end{eqnarray*}
We have $\lambda(D)\geq 0$, for $D$ effective, and so $\Delta_f(D)\geq 0$. 

 The formula $\Delta_f(f_*D)=d^N\Delta_f(D)$ follows immediately from the definition and, now that we have linearity, we can compute
\[\Delta_f(D)=d^{-N}\Delta_f(d^ND)=d^{-N}\Delta_f(f_*f^*D)=\Delta_f(f^*D).\]

For the final claim, suppose that $D$ is preperiodic and, without loss of generality, irreducible. Then for some $n\geq 0$ and $k\geq 1$, the divisors $f_*^{n+k}D$ and $f_*^nD$ are supported on the same irreducible hypersurface.
Comparing degrees, we have
\[f_*^{n+k}D=d^{(N-1)k}f_*^nD.\]
That in turn gives
 \[d^{(k+n)N}\Delta_f(D)=\Delta_f(f_*^{n+k}D)=\Delta_f(d^{(N-1)k}f_*^nD)=d^{(N-1)k+Nn}\Delta_f(D),\]
 by linearity, and so $\Delta_f(D)=0$.
\end{proof}

\begin{remark} 
In~\cite{mincrit} we defined a homogeneous escape rate $G_{F}(\Phi)$ for homogeneous forms $\Phi$ and affine maps $F(\mathbf{X})=A\mathbf{X}^d$. If we choose a lift $F$ for $f$, if $F_h$ is the homogenous part of $F$ (that is, with $\mathbf{b}$ replaced by $\mathbf{0}$), and $\Phi_h=\Phi(X_0, ..., X_{N-1}, 0)$, then we can check from the properties in Lemma~\ref{lem:basicdelta} and \cite[Lemma~10]{mincrit} that for $D$  defined by $\Phi=0$, we have
\[\Delta_f(D)=G_F(\Phi)-G_{F_h}(\Phi_h).\]
One virtue of the function $\Delta_f$ is that it does not depend on choosing models of $f$ and $D$.
\end{remark}

\begin{remark}\label{rem:integrals} 
In some sense it is more natural, in the case $K=\CC$, to work in terms of \[\lambda_{\mathrm{m}}(D)=\int \log\left|\frac{F(X_1, ..., X_{N+1})}{F(X_1, ..., X_{N}, 0)}\right|d\mu(\mathbf{X}),\]	
where $\mu$ is normalized Haar measure on the appropriate power of the unit circle, instead of $\lambda$ as defined above, naively in terms of the coefficients of a defining form. As noted in Lemma~\ref{lem:mult}, using inequality~\eqref{eq:mahler} (due to Mahler~\cite{mahler}), we have \[\lambda(D)=\lambda_{\mathrm{m}}(D)+O(\deg(D)),\]
with the implied constant depending only on $N$ and $d$. It then follows that, for fixed $D$,
\[\lambda(f_*^k D)=\lambda_{\mathrm{m}}(f_*^k D)+O_{d, N, D}(d^{k(N-1)}).\]
So the limit~\eqref{eq:greens} using either $\lambda_{\mathrm{m}}$ or $\lambda$ defines the same function $\Delta_f$.

Along similar lines, still over $\CC$, we can easily check, \emph{post hoc}, that
\[\Delta_f(D)=\int\log\left|\frac{F(X_1, ..., X_{N+1})}{F(X_1, ..., X_{N}, 0)}\right|d\mu_f(\mathbf{X})\] 
for any homogeneous form $F$ defining $D$, where $\mu_f$ is the invariant measure associated to $f$ (see, e.g.,~\cite[Lemma~11]{mincrit}). From this and~\cite[Theorem 3.2]{bj} we have
\begin{equation}\label{eq:lyapdef}\Delta_f(C_f)=L(f)-L(f|_H)-\log d,\end{equation}
from which we derive our main results over $\CC$.
\end{remark}

Lemma~\ref{lem:basicdelta} gives an estimate of the form
\[\Delta_f(D)=\lambda(D)+O(\deg(D)),\]
where the implied constant is explicit, but depends on $L$.
The next lemma shows that, once $\mu(D)$ is large enough, we can estimate $\Delta_f(D)$ from below in terms of $\lambda(D)$, with an error term that is much more uniform, depending only on the submatrix $A$. This submatrix represents the restriction of $L$ to the hyperplane at infinity, and so this can be seen as an assertion that all such maps with the same restriction to infinity are, near infinity, very similar (a philosophy which applies in general to regular polynomial endomorphisms).

\begin{lemma}\label{lem:basin} Let 
\[c_8 = \begin{cases}
 	\frac{2(N-1)(d^{N+1}-d+1)}{(d-d^{1/2})} & \text{if $|\cdot|$ is archimedean}\\
 	0 & \text{otherwise},
 \end{cases}
\]  
 suppose that 
 \begin{multline}\label{eq:breq}(d-1)\log^+\|\mathbf{b}\|>c_8(d^{1/2}-1)+c_3+c_5+\log\|A^{-1}\|+d\xi(A)\\+(2dN+1)\log^+|2|+(2N-2+dN(d^N-1))\log^+|d|,\end{multline} 
and suppose further that $D$ is non-zero effective divisor with
\begin{equation}\label{eq:mureq}\mu(D)\geq \log^+\|\mathbf{b}\|+c_8\Big(-1+\deg(D)^{1/2(N-1)}\Big)-\xi(A) .\end{equation}
Then
\[\Delta_f(D)\geq \lambda(D)-\frac{1}{d-1}\deg(D)(\log\|A^{-1}\|+\log^+|2N|)-\frac{N}{d^N-1}\log^+|2|\]	
\end{lemma}

\begin{proof}
Let $S$ be the set of effective divisors of degree at least 1 meeting the condition~\eqref{eq:mureq}, and let (for $x\in\RR^+$)
\[\psi(x)=dc_8(-1+x^{1/2(N-1)})-(d(d^N-1)+2)\log x\]
if $|\cdot|$ is archimedean, $\psi=0$ otherwise.
Note that $\psi(1)=0$, and we have chosen $c_8$ so that $\psi'(x)\geq 0$  for $x\geq 1$, whence $\psi(x)\geq 0$ for all $x\geq 1$.	 Similarly, let
\[\omega(x)=dc_8(-1+x^{1/2(N-1)})-c_8(-1+d^{1/2}x^{1/2(N-1)})-(d^{N+1}-d+1)\log x+c_8(d^{1/2}-1)\]
if $|\cdot|$ is archimedean, and $\omega =0$ otherwise,
and note that $\omega(x)\geq 0$ for $x\geq 1$.

Now, for $D\in S$ not containing the origin, we have from Lemma~\ref{lem:pushpullpower}
\begin{eqnarray*}
\mu(\phi_*D)&\geq & d\mu(D)-d(d^N-1)\log^+|\deg(D)|-2dN\log^+|2|-dN(d^N-1)\log^+|d|\\
&\geq &d\log^+\|\mathbf{b}\|+dc_8(-1+\deg(D)^{1/2(N-1)})-d\xi(A) \\&&-d(d^N-1)\log^+|\deg(D)|-2dN\log^+|2|-dN(d^N-1)\log^+|d|\\
&=&d\log^+\|\mathbf{b}\|+\psi(\deg(D))+2\log^+|\deg(D)|-d\xi(A)\\&&-2dN\log^+|2|-dN(d^N-1)\log^+|d|\\
&\geq&d\log^+\|\mathbf{b}\|+2\log^+|d^{N-1}\deg(D)|-2\log^+|d^{N-1}|-d\xi(A)\\&&-2dN\log^+|2|-dN(d^N-1)\log^+|d|\\
&=&\log^+\|\mathbf{b}\|+c_3+c_5+\log\|A^{-1}\|+2\log^+|\deg(\phi_*D)|\\&&+(d-1)\log^+\|\mathbf{b}\|-c_3-c_5-\log\|A^{-1}\|-d\xi(A)\\&&-(2N-2+dN(d^N-1))\log^+|d|-2dN\log^+|2|\\
&>&\log^+\|\mathbf{b}\|+c_3+c_5+\log\|A^{-1}\|+2\log^+|\deg(\phi_*D)|
\end{eqnarray*}
given~\eqref{eq:breq}.
It follows from this and Lemma~\ref{lem:keything} that
\begin{eqnarray*}
\mu(f_*D)&=&\mu(L_*\phi_*D)\\
&\geq&\mu(\phi_*D)-\log^+|\deg(\phi_*D)|-\log^+|2|-c_5-\log\|A^{-1}\|\\ 
&=&	 d\mu(D)-(d^{N+1}-d+1)\log^+|\deg(D)|-(N-1)\log^+|d|\\&&-(2dN+1)\log^+|2|-dN(d^N-1)\log^+|d|-c_5-\log\|A^{-1}\|\\
&\geq &d\log^+\|\mathbf{b}\|+dc_8(-1+\deg(D)^{1/2(N-1)})-d\xi(A)\\&&-(d^{N+1}-d+1)\log^+|\deg(D)|-(2dN+1)\log^+|2|\\&&-\left(dN(d^N-1)+(N-1)\right)\log^+|d|-c_5-\log\|A^{-1}\|\\
&= & \log^+\|\mathbf{b}\|+c_8(-1+\deg(\phi_*D)^{1/2(N-1)})-\xi(A)+\omega(\deg(D))\\&&-c_8(d^{1/2}-1)+(d-1)\log^+\|\mathbf{b}\|-(d-1)\xi(A)-(2dN+1)\log^+|2|\\&&-\left(dN(d^N-1)+(N-1)\right)\log^+|d|-c_5-\log\|A^{-1}\|\\
&>&\log^+\|\mathbf{b}\|+c_8(-1+\deg(\phi_*D)^{1/2(N-1)})-\xi(A)
\end{eqnarray*}
since $\omega(\deg(D))\geq 0$ and 
\begin{multline*}
(d-1)\log^+\|\mathbf{b}\|\geq c_8(d^{1/2}-1)+(d-1)\xi(A)+(2dN+1)\log^+|2|\\+\left(dN(d^N-1)+(N-1)\right)\log^+|d|+c_5+\log\|A^{-1}\|.	
\end{multline*}
In other words, $S$ is closed under the action of $f_*$.

On the other hand, since $\phi_*D$ satisfies the hypotheses of Lemma~\ref{lem:keything}, we also have
\begin{eqnarray*}
\lambda(f_*D)&=&\lambda(L_* \phi_* D)\\
&\geq &\lambda(\phi_*D)-\deg(\phi_*D)(\log\|A^{-1}\|+\log^+|2N|)-N\log^+|2|	\\
&=&d^N\lambda(D)-d^{N-1}\deg(D)(\log\|A^{-1}\|+\log^+|2N|)-N\log^+|2|
\end{eqnarray*}
As $S$ is closed under $f_*$, we can iterate this, giving
\[\frac{\lambda(f_*^kD)}{d^{kN}}\geq \lambda(D)-\frac{1-d^{-k}}{d-1}\deg(D)(\log\|A^{-1}\|+\log^+|2N|)-\frac{N}{d^N-1}\log^+|2|\]
for $D\in S$ and $k\geq 1$, from which the lower bound on $\Delta_f(D)$ follows.
\end{proof}

The next lemma, a lower bound on the relative escape rate of the critical divisor of $f$, is the main ingredient in the results of this paper.

\begin{lemma}\label{lem:localcritlower}
Let $f$ be as in~\eqref{eq:form}, let $C_f=\sum_{i=1}^N(d-1)H_i$ be the finite part of the critical divisor, and let
\begin{multline*}
	c_9=\max\Bigg\{0, \frac{1}{d-1}\log^+|2N|+\frac{N}{d^N-1}\log^+|2|-\frac{1}{N(d-1)}\log^+|N!|, \\	\frac{c_8(d^{1/2}-1)+c_3+c_5+(2dN+1)\log^+|2|+(2N-2+dN(d^N-1))\log^+|d|}{d-1}\Bigg\}.\end{multline*}
Then
\[\Delta_f(C_f)\geq \frac{d-1}{d}\log^+\|\mathbf{b}\|-\frac{1}{Nd}\lambda(A^{-1})-\xi(A)-\frac{d-1}{d}c_9.\]	
\end{lemma}

\begin{proof}
First, note that $\Delta_f(C_f)$, $\lambda(A)$, $\xi(A)$, and $c_9$ are all non-negative, and so our conclusion holds trivially if $\log^+\|\mathbf{b}\|=0$. We will, therefore, assume throughout that $\log^+\|\mathbf{b}\|=\log\|\mathbf{b}\|>0$.
 
Let $B_i=L_* H_i=(L^{-1})^*H_i$, noting that the finite part of the branch locus of $f$ is supported exactly on the $B_i$.
The hyperplanes $B_i$, for $1\leq i\leq N$ are defined by the linear forms $g_i$ whose coefficients make up the first $N$ rows of $L^{-1}$, which we recall is given by
\[L^{-1}=\begin{pmatrix}A^{-1} & -A^{-1}\mathbf{b} \\ \mathbf{0}  & 1	 \end{pmatrix}.\]
Now, note that
\[\log\|\mathbf{b}\|= \log\|AA^{-1}\mathbf{b}\|\leq \log\|A^{-1}\mathbf{b}\|+\log\|A\|+\log^+|N|,\]
and so there is some $i$ such that the $i$th entry $b_i'$ of $\mathbf{b}'=-A^{-1}\mathbf{b}$ satisfies
\begin{eqnarray*}
\log\|g_i\|
&\geq & \log|b_i'|\\&\geq& \log\|\mathbf{b}\|-\log\|A\|-\log^+|N|.
\end{eqnarray*}
On the other hand,
$\log\|g_i|_H\|\leq \log\|A^{-1}\|$,
and so we have
\[\lambda(B_i)\geq \log\|\mathbf{b}\|-\log\|A\|-\log\|A^{-1}\|-\log^+|N|=\log\|\mathbf{b}\|-\xi(A).\]
Note also $b_i'=g_{i, \deg(g_i)}$ in the notation of the definition of $\mu$, 
 and so we also have
\begin{eqnarray*}
\mu(B_i)&\geq& \log\|\mathbf{b}\|-\xi(A)\\&=&\log\|\mathbf{b}\|-\xi(A)+c_8\Big(-1+\deg(B_i)^{1/2(N-1)}\Big)	
\end{eqnarray*}
given that $\deg(B_i)=1$. Evidently, condition~\eqref{eq:mureq} in Lemma~\ref{lem:basin} is met, and so if~\eqref{eq:breq} is satisfied as well, then we have 
\begin{multline*}
\Delta_f(B_i)\geq \log\|\mathbf{b}\|-\frac{1}{N(d-1)}\lambda(A^{-1})+\frac{1}{N(d-1)}\log^+|N!|-\xi(A)\\-\frac{1}{d-1}\log^+|2N| -\frac{N}{d^N-1}\log^+|2|,	
\end{multline*}
which is stronger than
\begin{equation}
\label{eq:this thing here}
\Delta_f(B_i)\geq \log\|\mathbf{b}\|-\frac{1}{N(d-1)}\lambda(A^{-1})-\frac{d}{d-1}\xi(A)-c_9.
\end{equation}
If, on the other hand, we fail to meet~\eqref{eq:breq}, then 
\begin{eqnarray*}
(d-1)\log\|\mathbf{b}\|-\frac{1}{N}\lambda(A^{-1})-d\xi(A) 
&\leq &(d-1)\log\|\mathbf{b}\|-\log\|A^{-1}\|-d\xi(A)\\
&\leq& c_8(d^{1/2}-1)+c_3+c_5+(2dN+1)\log^+|2|\\&&+(2N-2+dN(d^N-1))\log^+|d|\\& \leq& (d-1)c_9,\end{eqnarray*}
in which case~\eqref{eq:this thing here} is true simply because $\Delta_f(B_i)\geq 0$.

Inequality~\eqref{eq:this thing here}, combined with the non-negativity of $\Delta_f$, gives
\begin{eqnarray*}
\Delta_f(C_f)&=&\sum_{j=1}^N(d-1)\Delta_f(H_j)\\
&\geq & (d-1)\Delta_f(H_i)\\&=&\frac{d-1}{d^N}\Delta_f(f_*H_i)\\
&=&\frac{d-1}{d^N}\Delta_f(d^{N-1}B_i)\\
&\geq & \frac{d-1}{d}\log^+\|\mathbf{b}\|-\frac{1}{Nd}\lambda(A^{-1})-\xi(A)-\frac{d-1}{d}c_9
\end{eqnarray*}
\end{proof}

Although Lemma~\ref{lem:localcritlower} is the key ingredient in our main results, one might wish to record the corresponding bound in the other direction, an immediate consequence of results already shown.

\begin{lemma}\label{lem:trivial} 
There is a bound of the form
\begin{multline*}\Delta_f(C_f)\leq N(N+2)\log^+\|\mathbf{b}\|+(N+1)\lambda(A)+N\log^+|(N+1)!|+\log^+|N!|\\+N\log^+|4N(N+1)|
+\left(4N^2d+\frac{2N-1}{d^N-1}\right)\log^+|2|.
 \end{multline*}\end{lemma}

\begin{proof}
Note that $C_f$ is defined by the monomial equation
\[X_1^{d-1}\cdots X_N^{d-1}=0,\]
and so $\lambda(C_f)=0$. By Lemma~\ref{lem:basicdelta}, or more precisely~\eqref{eq:deltaineq} therein,
we have
\begin{multline*}\Delta_f(C_f)\leq \frac{\deg(C_f)}{d-1}\Big(\log\|L\|-\log\|A\|+\log^+|4N(N+1)|\\
+\lambda(L)+4Nd\log^+|2|\Big)+\left(\frac{2N-1}{d^N-1}\right)\log^+|2|.
 \end{multline*}
 Now, since  $A\in \SL_N(K)$, we have $-\frac{1}{N}\log^+|N!|\leq \log\|A\|$, and so 
 \[\log\|L\|-\log\|A\|\leq \log^+\|A\|+\log^+\|\mathbf{b}\|-\log\|A\|\leq \log^+\|\mathbf{b}\|+\frac{1}{N}\log^+|N!|.\]
 Also, we have
 \begin{eqnarray*}
 	\lambda(L) & = & (N+1)\log\|L\|+\log^+|(N+1)!|\\
 	&\leq & (N+1)\log^+\|\mathbf{b}\|+\frac{N+1}{N}\log^+|N!|+(N+1)\log\|A\|+\log^+|(N+1)!|\\
 	&=&(N+1)\log^+\|\mathbf{b}\|+\frac{N+1}{N}\lambda(A)+\log^+|(N+1)!|.
 \end{eqnarray*}
 The claim follows, since $\deg(C_f)=N(d-1)$. 
\end{proof}

Thus, with a view to fixing $A$, we have
\[\frac{d-1}{d}\log^+\|\mathbf{b}\|-O_A(1)\leq \Delta(C_f)\leq N(N+2)\log^+\|\mathbf{b}\|+O_A(1),\]
which proves Theorem~\ref{th:proper}. Indeed, the error terms can be made explicit in terms of $\lambda(A)$, $\lambda(A^{-1})$, and $\xi(A)$, all   non-negative functions on $\SL_N(K)$, which are  plurisubharmonic and continuous when $K=\CC$. The following proposition also proves Corollary~\ref{cor:proper}, recalling~\eqref{eq:lyapdef}.

\begin{prop}
In the case $K=\CC$, the map \[\SL_{N+1}(\CC)\times \CC^N\to \SL_{N+1}(\CC)\times \RR\]
by
\[(A, b)\mapsto (A, \Delta_f(C_f))\]
is continuous, plurisubharmonic, and proper. In particular, if 
\[\mathcal{M}=\{(A, \mathbf{b})\in \SL_{N+1}(\CC)\times \CC^N:L(f)=L(f|_H)+\log d\},\]
then the projection $\pi:\mathcal{M}\to \SL_{N+1}(\CC)$ is proper.
\end{prop}

\begin{proof}
Note that, by~\eqref{eq:lyapdef}, the set $\mathcal{M}$ is equivalently defined as the set of pairs for which $\Delta_f(C_f)=0$.

For any $k\geq 0$, the function
\[(A, \mathbf{b})\mapsto \frac{\lambda(f^k_*C_f)}{d^{kN}}\]
is continuous and plurisubharmonic, since there is a homogeneous form defining $f^k_*C_f$ whose coefficients are polynomials in the entries of $A$ and $\mathbf{b}$. But from Lemma~\ref{lem:basicdelta}, these functions converge uniformly on compact subsets to $\Delta_f(C_f)$, and so $f\mapsto \Delta_f(C_f)$ is continuous and plurisubharmonic. This part of the result can also be accessed by work of Berteloot and Basanelli~\cite[Section~1.4]{bb} (see also~\cite{ds}). 

Now, on any compact $E\subseteq \SL_{N+1}(\CC)\times\RR$ the functions $\lambda(A)$ and $\xi(A)$ are bounded, and so Lemma~\ref{lem:localcritlower} gives, for $(A, \mathbf{b})\in \mathcal{M}$,
\[\log^+\|\mathbf{b}\|\leq \frac{d}{d-1}\Delta_f(C_f)+O_E(1)=O_E(1),\]
for $\Delta_f(C_f)$ in the projection of the second coordinate of $E$. Since $\SL_{N+1}(\CC)\times \CC^N\to \SL_{N+1}(\CC)\times \RR$ is continuous,  it is also proper.
\end{proof}

We end with a generalization of the observation that, if $z^d+c$ is PCF, then $c$ is an algebraic integer. We recall that a morphism $f:\PP^N_K\to\PP^N_K$, with $K$ complete with respect to a non-archimedean absolute value, has good reduction if and only if $f$ extends to a scheme morphism $\overline{f}:\PP^N_R\to \PP^N_R$ over the ring of integers $R$. Equivalently, if we choose homogeneous forms representing $f$, whose coefficients are integral and at not all contained in the maximal ideal, then $f$ has good reduction if and only if the resultant of these homogeneous forms is a unit.

\begin{prop}\label{prop:goodred}
Let $K$ be an algebraically closed field, complete with respect to a non-archimedean absolute value which is not $p$-adic for any $p\leq \max\{d, N!\}$, with ring of integers $R$. Then any PCF map of the form~\eqref{eq:form} with $A\in\SL_N(R)$ has good reduction.
\end{prop}

Good reduction has various dynamical consequences. For example, if $f:\PP^N\to\PP^N$ has good reduction, and $f(P)=P$, then any eigenvalue $\lambda$ of the action of $f$ on the tangent space at $P$ satisfies $|\lambda|\leq 1$ (so, in a strong sense, periodic points are non-repelling).

\begin{proof}[Proof of Proposition~\ref{prop:goodred}]
For $A\in\SL_N(R)$, we claim that $f$ has good reduction if and only if the entries of $\mathbf{b}$ are integral. To see this, note that if the entries of $\mathbf{b}$ are integral, then the  entries of the matrix $L=\begin{pmatrix}A & \mathbf{b} \\ \mathbf{0} & 1\end{pmatrix}$ are integral, and are the coefficients of some homogeneous forms $F_i(\mathbf{X})=\sum_{j=1}^{N+1} L_{i, j}X_{j-1}^d$ defining $f$. If all entries are in the maximal ideal, then we could not have $\det(A)=1$, and so $f$ has good reduction if and only if the resultant of these forms is a unit. But by \cite[Theorem~3.13, p.~399]{lang} these homogeneous forms have resultant
\[\det(L)^{d^N}=\det(A)^{d^N}=1.\]

If, on the other hand, the entries of $\mathbf{b}$ are non-integral, then for $\pi$ a uniformizer of the absolute value, there is some $\epsilon>0$ so that the entries of $\pi^\epsilon L$ are integral, and at least one is a unit. We have $\det(\pi^\epsilon L)=\pi^{(N+1)\epsilon}$ not a unit, and so by the same argument as above, $f$ has bad reduction.

	Now, given our assumptions on the absolute value, we have
	$\xi(A)=0$ and $\log\|A^{-1}\|=0$, and Lemma~\ref{lem:localcritlower} gives
	\[\Delta_f(C_f)\geq \frac{d-1}{d}\log^+\|\mathbf{b}\|.\]
	If $f$ is PCF, and hence $\Delta_f(C_f)=0$, it follows that $\log^+\|\mathbf{b}\|=0$, and so the entries of $\mathbf{b}$ are integral. 
\end{proof}

\section{Global results}\label{sec:global}

We now change context so that $K$ is a field with a collection of inequivalent absolute values $M_K$ with weights $n_v$ such that the product formula holds for $\alpha\in K^*$, that is,
\begin{equation}\label{eq:prodfla}\sum_{v\in M_K}n_v\log|\alpha|_v=0.\end{equation}
Our main example is when $K$ is a number field, $M_K$ is the standard set of absolute values, and $n_v=[K_v:\QQ_v]/[K:\QQ]$. For any absolute value $|\cdot|_v$ on $K$ we may apply results from the previous section to an algebraic closure of a completion of an algebraic closure of $K$, with respect to $v$, and all quantities thereby obtained now acquire a subscript $v$.

For a divisor $D$ on $\PP^N_K$ defined by the vanishing of the homogeneous form $F(\mathbf{X})\in K[X_1, .., X_{N+1}]$, let
\[h(D)=\sum_{v\in M_K}n_v\log\|F\|_v,\]
that is, let the height of $D$ be the height of the tuple of coefficients as a point in the appropriate dual projective space. Note that, by~\eqref{eq:prodfla}, this definition is independent of the choice of form defining $D$, while~\eqref{eq:mahler} can be used to relate this height to the height used by Philippon~\cite{philippon}, which we used in~\cite{pnheights},  and then to that of Faltings~\cite{faltings} (see~\cite{soule}).

 Then we see immediately that for $D$ not containing $H$,
\[\sum_{v\in M_K}n_v\lambda_v(D)=h(D)-h(D|_H).\]
Writing $\hat{h}_f(D)$ for the canonical height of $D$ relative to $f$, so \[\hat{h}_f(D)=\lim_{k\to\infty} \frac{h(f_*^kD)}{d^{kN}},\] we then have
\[\sum_{v\in M_K}n_v\Delta_{f, v}(D)=\hat{h}_f(D)-\hat{h}_{f|_H}(D|_H)\]
(see also~\cite{pnheights, mincrit}).
Note that in~\cite{pnheights}, a different naive height was used on divisors, but since the heights differ by at most $O(\deg(D))$, the canonical height is the same (see Remark~\ref{rem:integrals}). Also note the one subtlety here, that $(f_*D)|_H = d(f|_H)_* D|_H$, so that $\hat{h}_{f|_H}((f_*D)|_H)=d^N\hat{h}_{f|_H}(D|_H)$, despite $\dim(H)=N-1$.

Note that since $f^*H=dH$, it follows readily that for $C_f$ as above, $C_f+(d-1)H$ is the ramification divisor of $f$, and $\hat{h}_{\mathrm{crit}}(f)=\hat{h}_f(C_f)$. In particular, we have
\[\sum_{v\in M_K}n_v\Delta_{f,v}(C_f)=\hat{h}_{\mathrm{crit}}(f)-\hat{h}_{\mathrm{crit}}(f|_H).\]

We may now proceed with the proofs of the global results.
\begin{proof}[Proof of Theorem~\ref{th:main}]
Let $K$ be a number field, and let $f$ be as in~\eqref{eq:form}, with coefficients in $K$. At each place $v$ of $K$, with subscripts denoting dependence on the corresponding absolute value, we have from Lemma~\ref{lem:localcritlower} that
\[\Delta_{f, v}(C_f)\geq \frac{d-1}{d}\log^+\|\mathbf{b}\|_v-\frac{1}{Nd}\lambda_v(A^{-1})-\xi_v(A)-\frac{d-1}{d}c_{9, v}.\]	
Summing over all places, we obtain the desired bound once we note that
\begin{gather*}
	\sum_{v\in M_K}n_v\log^+\|\mathbf{b}\|_v=h_{\PP^N}(\mathbf{b})\\
	\sum_{v\in M_K}n_v\lambda_v(A)=Nh_{\PGL_{N+1}}(A)+\log N!\\
	\sum_{v\in M_K}n_v\xi(A)=h_{\PGL_{N+1}}(A)+h_{\PGL_{N+1}}(A^{-1})+\log N\\
	\intertext{and}
	h_{\PGL_{N+1}}(A^{-1})\leq (N-1)h_{\PGL_{N+1}}(A)
\end{gather*}
while $\sum_{v\in M_K}n_vc_{9, v}$ is some explicit constant depending just on $N$ and $d$.
Note that this last upper bound contains the sum $\sum_{p\leq d}\frac{\log p}{p-1}$, a sum over primes, which can be explicitly bounded above in terms of $d$ using estimates of Rosser and Schoenfeld~\cite{rs}.

Similarly, the upper bound on $\hat{h}_{\mathrm{crit}}(f)-\hat{h}_{\mathrm{crit}}(f|_H)$ comes from summing the estimates in Lemma~\ref{lem:trivial} over all places.
\end{proof}

Note that the terms $\hat{h}_{\mathrm{crit}}(f)$ and $\hat{h}_{\mathrm{crit}}(f|_H)$ in Theorem~\ref{th:main} are independent of choice of coordinates, while the terms $h(\mathbf{b})$ and $h(A)$ are not. Indeed, it is possible to take $h(\mathbf{b})\to \infty$ within a conjugacy class, which might seem troubling at first for a lower bound on an invariant of the class, but note that this would result in the error term increasing as well.

\begin{proof}[Proof of Theorem~\ref{th:rigid}]
	Suppose our putative algebraic family is defined over the variety $V/k$, and let $K=k(V)$ be the function field, so that we may think of $f$ as a single map with coefficients in $K$. There exists a set $M_K$ of inequivalent absolute values such that the elements of height zero are precisely the constants (namely, we can take $M_K$ to be the collection of absolute values corresponding to vanishing of functions on irreducible divisors on any normal, projective variety $V'$ birational to $V$~\cite[Lemma~1.4.10, p.~12]{bg}).
	
	All of these absolute values are non-archimedean, and none are $p$-adic, and since $A$ is constant we have by Proposition~\ref{prop:goodred} that the entries of $\mathbf{b}$ are integral. In other words, given any irreducible divisor on $V'$, the functions $b_i$ do not have a pole along $V'$, and since the divisor was arbitrary, the $b_i$ are all constant.
\end{proof}

Finally, we note that the results in the previous sections allow for explicit estimates on the difference between the canonical height and the naive height of a divisor. Such results appear in~\cite{mincrit}, but here (and with regular polynomial endomorphisms in general) it seems to make more sense to think in terms of relative quantities.
\begin{prop}
	Let $D$ be an effective divisor on $\PP^N$, and $f$ as in~\eqref{eq:form}. Also, write
	\[h_{\mathrm{rel}}(D)=h(D)-h(D|H)\]
	and
	\[\hat{h}_{\mathrm{rel}, f}(D)=\hat{h}_f(D)-\hat{h}_{f|_H}h(D|H)\]
	Then
	\[\hat{h}_{\mathrm{rel}, f}(D)=h_{\mathrm{rel}}(D)+\deg(D)O_{d, N}(h_{\PGL_{N+1}}(L)+1).\]
\end{prop}

\begin{proof}
Similar to the other results in this section, this is just a matter of summing~\eqref{eq:deltaineq} over all places.
\end{proof}

\section{The cases $d>N^2-N+1$ and $N=2$}\label{sec:p2}

Here we make a few remarks on cases in which the relative results in the introduction become absolute, largely by leveraging the results in~\cite{mincrit}.

\begin{prop}\label{prop:p2}
	Let $d>N^2-N+1$ or $N=2$. Then the PCF maps of the form~\eqref{eq:form} are a set up bounded height, up to conjugation.
\end{prop}

\begin{proof}
	If $f$ of the form~\eqref{eq:form} is PCF, then so is $f|_H$, which is a minimally critical endomorphism in the sense of~\cite{mincrit}. If $d>N^2-N+1=(N-1)^2+(N-1)+1$, then the main result of~\cite{mincrit} shows that $f|_H$ is conjugate to a map of the form $B\mathbf{X}^d$ with $h_{\PGL_{N}}(B)$ bounded in terms of $d$ and $N$. We can extend this change of coordinates to $\PP^N$ and choose a lift of $B$ to $\SL_N$, and thereby replace $f$ by a map $f(\mathbf{X})=A\mathbf{X}^d+\mathbf{b}$ with $h(A)$ bounded. But now Theorem~\ref{th:main} gives us that $h(\mathbf{b})$ is bounded as well (in terms of $d$ and $N$).
	
	In the case $N=2$ we may extend this to $d=2, 3$ by the main result of~\cite{bijl}. Here, $f|_H$ is a minimally critical (bicritical) endomorphism of $\PP^1$ which is PCF, and hence has bounded moduli height. It is not \emph{a priori} obvious that this map will be conjugate to something of the form $B\mathbf{X}^d$ with $B\in \PGL_2$ of bounded height, but this follows from~\cite[Lemma~15]{mincrit} and~\cite[Lemma~6.32, p.~102]{barbados}. The rest of the argument is now the same
\end{proof}

\begin{remark}
There are, of course, a bevy of examples of PCF endomorphisms of the form~\eqref{eq:form} with $\mathbf{b}=\mathbf{0}$, but we expect examples with $\mathbf{b}\neq \mathbf{0}$ to be quite rare. As such, it would be interesting to compute exhaustive lists of examples defined over $\QQ$, say, in the case $N=2$, which brings us to the question of how explicit the bounds in Proposition~\ref{prop:p2} can be made.

The bounds for $h(\mathbf{b})$ can be made completely explicit,  in terms of $h(A)$, by a careful tracing through the proof of Theorem~\ref{th:main}.   In the case $d=2$, bounds for $h(A)$ are made concrete in~\cite{bijl}, and so the exhaustive list  implied by Proposition~\ref{prop:p2} could actually be computed (but not easily). In the case $d\geq 3$, the results in~\cite{bijl} do not imply anything quite so explicit, but we note that a more direct argument gives effective constants when $d\geq 4$ (see~\cite{mincrit}).
\end{remark}

\begin{prop}
	Let $d>N^2-N+1$ or $N=2$. Then there are no algebraic families of PCF maps of the form~\eqref{eq:form} over $\CC$.
\end{prop}

\begin{proof}\label{prop:p2c}
	If $d>N^2-N+1$, then again the results of~\cite{mincrit} apply to the restriction $f|_H$. So if $f$ is a PCF family, then $f|_H$ is also a PCF family, and by~\cite[Theorem~3]{mincrit} must be isotrivial. Extending this change of coordinates to $\PP^N$, we may replace $f$ by a conjugate family (perhaps after a finite extension of the function field) so that $f|_H$ is constant. Theorem~\ref{th:rigid} now applies to show that all coefficients of $f$ are constant. If $N=2$ the argument is the same, except now in the case $d=2, 3$ we must use Thurston's result to conclude that the family $f|_H$ of PCF endomorphisms of $\PP^1$ is isotrivial.
\end{proof}

In positive characteristic, we may still apply the results of~\cite{mincrit} and prove a version of Proposition~\ref{prop:p2c} when $d>N^2-N+1$. In the case $N=2$, $d=3$ we are out of luck, but the remaining case $N=d=2$ can be treated in odd characteristic by the exact same proof, and an appeal to the rigidity of PCF quadratric endomorphisms of $\PP^1$ in odd characteristic~\cite{bijl}.

\begin{prop}\label{th:p2rigid}
In characteristic $p\neq 2$, there are no algebraic families of quadratic PCF maps $f:\PP^2\to\PP^2$ of the form~\eqref{eq:form}.
\end{prop}

\end{document}